\documentclass[11pt,a4paper]{article}
\usepackage{silence}
\usepackage{filecontents}
\usepackage[utf8]{inputenc}
\usepackage[margin=1in]{geometry}
\usepackage{xcolor}
\usepackage{graphicx}
\usepackage{bm}
\usepackage{bbm}
\usepackage{caption}
\usepackage{subcaption}
\usepackage{amsmath}
\usepackage{amsthm}
\usepackage{amssymb}
\usepackage{amsfonts}
\usepackage{mathtools}
\usepackage{xparse}
\usepackage{epstopdf}
\usepackage[colorlinks=true,citecolor=blue]{hyperref}
\usepackage{cleveref}

\usepackage{cite}

\usepackage{setspace}
\DeclareDocumentCommand\sobolev{m m o} {H^{#1}_{#3}(#2 \IfNoValueF{#3})}
\DeclareDocumentCommand\lp{m m o} {L^{#1}\IfNoValueF{#3}{_{#3}}\left(#2\right)}
\DeclareDocumentCommand\cont{o m o} {C\IfNoValueF{#1}{^{#1}}(#2\IfNoValueF{#3}{;#3})}
\DeclareDocumentCommand\contc{o m o} {C_c\IfNoValueF{#1}{^{#1}}(#2\IfNoValueF{#3}{;#3})}
\DeclareDocumentCommand\norm{s m o} {\IfBooleanTF{#1}{\|#2\|}{\left\|#2\right\|}\IfNoValueF{#3}{_{#3}}}
\DeclareDocumentCommand\seminorm{s m o} {\IfBooleanTF{#1}{\|#2\|}{\left\|#2\right\|}\IfNoValueF{#3}{_{#3}}}
\DeclareDocumentCommand\ip{s m m o} {\IfBooleanTF{#1}{( #2,#3 )}{\left( #2,#3 \right)}\IfNoValueF{#4}{_{#4}}}
\DeclareDocumentCommand\eip{s m m o} {\IfBooleanTF{#1}{\langle #2,#3 \rangle}{\left\langle #2,#3 \right\rangle}\IfNoValueF{#4}{_{#4}}}
\DeclareDocumentCommand\abs{s m o} {\IfBooleanTF{#1}{|#2|}{\left|#2\right|}\IfNoValueF{#3}{_{#3}}}
\DeclareDocumentCommand\eucnorm{s m o} {\IfBooleanTF{#1}{|#2|}{\left|#2\right|}\IfNoValueF{#3}{_{#3}}}
\DeclareDocumentCommand\wass{s o m m} {W_{\IfNoValueF{2}{#2}}\IfBooleanTF{#1}{(#3, #4)}{\left(#3, #4\right)}}

\DeclareMathOperator{\Law}{Law}
\DeclareMathOperator{\e}{e}

\DeclareMathOperator*{\trace}{tr}

\newcommand{\dummy}{\mathord{\color{black!33}\bullet}}%
\newcommand{\expect}{\mathbf{E}}

\newcommand{\mat}[1]{\mathit{#1}}
\newcommand{\nat}{\mathbf N}

\newcommand{\real}{\mathbf R}

\newcommand{\vect}[1]{\boldsymbol{\mathbf #1}}
\newcommand{\grad}{\nabla}
\renewcommand{\d}{\mathrm d}
\newcommand{\one}{\mathbbm 1}

\makeatletter
\DeclareDocumentCommand \derivative{s m o m}{%
    \def\@der{\IfBooleanTF{#1}{\mathrm{d}}{\partial}}
    \def\@default{%
        \mathchoice{%
                \frac{%
                    \@der\ifnum\pdfstrcmp{#2}{1}=0\else^{#2}\fi {\IfNoValueTF{#3}{}{#3}}
                }{%
                    \@for\@token:={#4}\do{\@der \@token}
                }
            } {%
                \@for\@token:={#4}\do{\@der_{\@token}\ifnum\pdfstrcmp{#2}{1}=0\else^{#2}\fi} \IfNoValueTF{#3}{}{#3}
            } {} {}
    }
    \IfBooleanTF{#1}{\IfNoValueTF{#3}{\@default}{%
                #3%
                \ifnum\pdfstrcmp{#2}{1}=0'\else%
                \ifnum\pdfstrcmp{#2}{2}=0''\else%
                \ifnum\pdfstrcmp{#2}{3}=0'''\else%
                \ifnum\pdfstrcmp{#2}{4}=0^{(iv)}\else^{(#2)}\fi\fi\fi\fi
            }
        }{\@default}
}
\makeatother

\definecolor{darkred}{rgb}{0.5,0,0}
\definecolor{darkgreen}{rgb}{0,0.5,0}
\definecolor{darkblue}{rgb}{0,0,.5}

\theoremstyle{plain}

\newtheorem{lemma}{Lemma}[section]

\newtheorem{proposition}{Proposition}[section]

\newtheorem{remark}{Remark}[section]
\numberwithin{equation}{section}

\crefname{lemma}{Lemma}{Lemmata}
\crefname{remark}{Remark}{Remarks}
\crefname{proposition}{Proposition}{Propositions}
\crefname{section}{Section}{Sections}
\crefname{equation}{}{}
\Crefname{equation}{Equation}{Equations}

\title{Wasserstein stability estimates for covariance-preconditioned Fokker--Planck equations}
\author{J. A. Carrillo, U. Vaes}

\newcommand{\orcid}[1]{\href{https://orcid.org/#1}{\includegraphics[width=.4cm]{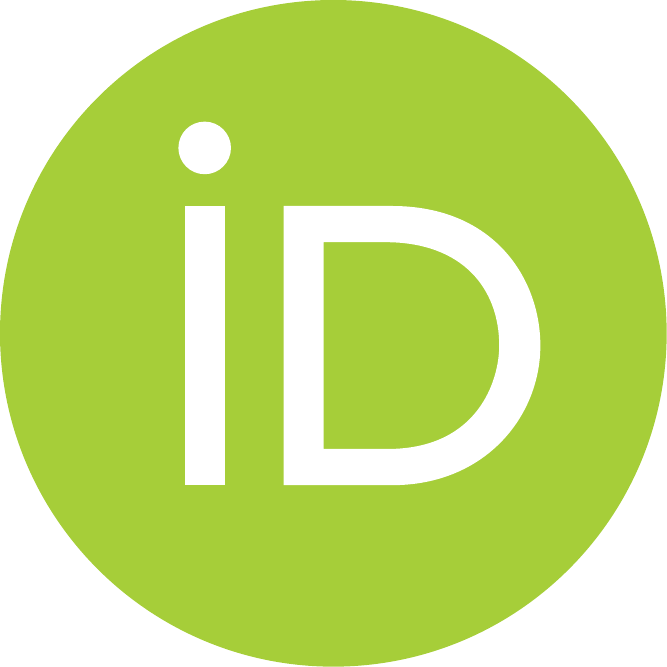}}}
\author{%
J. A. Carrillo%
\hspace{1mm}\orcid{0000-0001-8819-4660}%
\hspace{1cm} U. Vaes%
\hspace{1mm}\orcid{0000-0002-7629-7184} \\[.1cm]
{\small Department of Mathematics, Imperial College London} \\[.1cm]
{\small \tt{\{carrillo,u.vaes13\}@imperial.ac.uk}}
}

\begin{document}
\maketitle

\begin{center}
\begin{minipage}[b]{.9\textwidth}
    \setlength{\parskip}{6pt}
    \centerline{\textbf{Abstract}}
    We study the convergence to equilibrium of the mean field PDE associated with the derivative-free methodologies for solving inverse problems that are presented in~\cite{HV18,2019arXiv190308866G}.
    We show stability estimates in the euclidean Wasserstein distance for the mean field PDE by using optimal transport arguments.
    As a consequence, this recovers the convergence towards equilibrium estimates in~\cite{2019arXiv190308866G} in the case of a linear forward model.
\end{minipage}
\end{center}

\section{Introduction}%
\label{sec:introduction}

In this paper, we are concerned with the nonlocal Fokker--Planck equation
\begin{align}
    \label{eq:mean-field_equation}%
    \derivative{1}[f]{t}(\vect u, t) = \grad \cdot \Big( \mathcal C(f_t)  \left( \grad \Phi_R(\vect u; \vect y) \, f(\vect u, t)  + \sigma \grad f(\vect u, t) \right) \Big), \qquad & \vect u \in \real^d, t \in \real_{\geq 0},
\end{align}
where $\sigma > 0$, $f_t = f(\dummy, t)$,
$\mathcal C$ is the covariance operator defined by
\begin{align*}
    \mathcal C(f) = \int_{\real^d} \left(\vect u - \mathcal M(f)\right) \otimes \left(\vect u - \mathcal M(f)\right) \, f(\vect u) \, \d \vect u, \qquad \text{with } \mathcal M(f) = \int_{\real^d} \vect u \, f(\vect u) \, \d \vect u,
\end{align*}
and $\Phi(\dummy; \vect y)$ is a functional of the form
\begin{equation}
    \label{eq:least_squares_functional}%
    \Phi_R(\vect u; \vect y) = \frac{1}{2} \eucnorm*{\vect y -  \mathcal G(\vect u)}[\mat \Gamma]^2 + \frac{1}{2}\eucnorm*{\vect u}[\mat \Gamma_0]^2 =: \Phi(\vect u, \vect y) + \frac{1}{2}\eucnorm*{\vect u}[\mat \Gamma_0]^2.
\end{equation}
Here $\mathcal G: \real^d \rightarrow \real^{K}$ is a function that we will refer to as the \emph{forward model},
in view of the link with Bayesian inverse problems discussed below,
$\vect y \in \real^d$ is a given vector of \emph{observations} and $\mat \Gamma \in \real^{K \times K}, \mat \Gamma_0 \in \real^{d \times d}$ are symmetric, positive definite matrices.
We employed the notation $\eucnorm*{\dummy}[\mat \Gamma] := \eucnorm*{\mat \Gamma^{-\frac{1}{2}} \dummy}$,
where $\eucnorm{\dummy}$ is the usual euclidean norm.

Throughout this paper,
we restrict our attention to the case where $\mathcal G$ is a linear mapping and we write $\mathcal G(\vect u) = \mat G \vect u$,
with $\mat G \in \real^{K \times d}$.
We will assume that the matrix $\mat \Gamma_0^{-1} + \mat G^T \mat \Gamma^{-1} \mat G =: \mat B^{-1}$ is nonsingular,
so that the regularized least squares misfit $\Phi_R$,
given by \cref{eq:least_squares_functional},
admits the unique minimizer $\vect u_0 = \mat B \mat G^T \mat \Gamma^{-1} \vect y$.
Our main result is that,
if $f^1$ and $f^2$ are the solutions of \cref{eq:mean-field_equation} associated with the initial conditions $f^1_0$ and $f^2_0$, respectively,
then a stability estimate of the following form holds:
\begin{equation}%
    \label{eq:main_result}%
    \wass[2]{f^1_t}{f^2_t} \leq C(f^1_0, f^2_0; \mat G, \mat \Gamma) \, \gamma(t) \, \wass[2]{f^1_0}{f^2_0},
\end{equation}
where $C(\dummy_1, \dummy_2; \mat G, \mat \Gamma)$ depends only on the first two moments of $\dummy_1$ and $\dummy_2$
and the function $\gamma(t)$ converges to zero as $t \to \infty$ exponentially when $\sigma > 0$ and algebraically when $\sigma = 0$.
Here and in the rest of the paper, we employed the notation $f^{i}_t = f^i(\dummy, t)$, $i = 1, 2$.
If $\sigma > 0$,
then by taking one solution in \cref{eq:main_result} to be the equilibrium Gaussian
one recovers the equilibration estimate obtained in~\cite{2019arXiv190308866G}.
As a byproduct of our analysis,
we deduce the algebraic convergence of the solution towards a Dirac delta at
$\vect u_0$ when $\sigma=0$, i.e.\ to the solution of the Bayesian inverse problem,
generalizing to the mean field PDE the estimates obtained for a related particle system in~\cite{MR3654885}
and answering fully the equilibration open problem discussed in \cite{HV18} for the linear forward model.

We now turn our attention to the connection of the PDE~\eqref{eq:mean-field_equation} to mean field descriptions of the Ensemble Kalman methods for the Bayesian inverse problem.
The Fokker--Planck equation~\eqref{eq:mean-field_equation} can be linked to the inverse problem of finding $\vect u \in \real^d$ from an \emph{observation} $\vect y \in \real^K$ where
\begin{equation}
    \label{eq:inverse_problems}%
    \vect y = \mathcal G(\vect u) + \vect \eta.
\end{equation}
Here $\vect \eta$ is a random variable assumed to have Lebesgue density $\rho$.
In the Bayesian approach to inverse problems~\cite{MR3839555,doi:10.1002/wcc.535},
a probability measure called the \emph{prior} is placed on $\vect u$.
If we assume that this measure also has a density $\rho_0$ and that $\vect u$ is independent of $\vect \eta$,
then $(\vect u, \vect y)$ is a random variable with density $\rho(\vect y - \mathcal  G(\vect u)) \, \rho_0(\vect u)$.
The posterior density of $\vect u| \vect y$ (i.e. of $\vect u$ given an observation $\vect y$) is then given by the normalized probability density
\begin{equation}
    \label{eq:posterior_distribution}%
    \rho^{\vect y}(\vect u) = \frac{\rho(\vect y - \mathcal  G(\vect u)) \, \rho_0(\vect u)}{\int_{\real^d} \rho(\vect y - \mathcal  G(\vect u)) \, \rho_0(\vect u) \, \d \vect u}.
\end{equation}
In the particular case where $\rho$ and $\rho_0$ are the densities of Gaussians $\mathcal N(0, \mat \Gamma)$ and $\mathcal N(0, \mat \Gamma_0)$, respectively,
$\rho^{\vect y} \propto \e^{-\Phi_R(\vect u; \vect y)}$, where $\Phi_R$ is given by \cref{eq:least_squares_functional}.
We make this assumption below.

In~\cite{MR3041539},
the authors proposed to solve the inverse problem~\eqref{eq:inverse_problems} by applying a state-estimation method,
or filter, to the following artificial dynamics on $\real^d \times \real^K$ and associated observational model,
where we denote by $\vect u$ the first $d$ components of $\vect z$:
\[
    \vect z_{n+1} = \Xi(\vect z_n), \qquad \Xi(\vect z) = \begin{pmatrix} \vect u \\ \mathcal G(\vect u) \end{pmatrix}, \qquad \vect y_{n+1} = \begin{pmatrix} 0 & I \end{pmatrix} \vect z_{n+1} + \vect \eta_{n+1},
\]
where $\{\vect \eta_n\}_{n \in \nat}$ are i.i.d.\ $\mathcal N(0, h^{-1} \mat \Gamma)$ random variables.
If the observed data in the dynamics is fixed at the observation of the Bayesian inverse problem $\vect y$ for all steps,
then the $\vect u$-marginal of the posterior distribution at iteration $n$ has density
\[
    \rho_n(\vect u) \propto  \exp(-n h \Phi(\vect u; \vect y)) \, \rho_0(\vect u),
\]
which can be obtained by repeatedly applying the reasoning that led to \cref{eq:posterior_distribution}.
It is clear that this iteration will lead to a concentration of the mass of $\rho_i$ at minimizers of the (non-regularized) least squares functional $\Phi$.
We also remark that the posterior $\rho_n$ coincides with the posterior $\rho^{\vect y}$ of the inverse problem when $n h = 1$,
a fact that can be exploited to produce approximate samples of the posterior~\cite{Chen2012}.

If the prior $\rho_0$ is Gaussian and the forward model $\mathcal G$ is linear,
then the posteriors $\{\rho_n\}_{n \in \nat}$ can be captured exactly by a Kalman filter.
However, when the dimension of the state space is large,
which is often the case in scientific and engineering applications,
the Kalman filter is computationally expensive and a particle-based method such as the Ensemble Kalman filter (EnKF) becomes preferable.
This approach is also more general than the Kalman filter,
because it does not require that the forward model be linear.
The ensemble members $U = \{\vect u^{(j)}\}_{j=1}^J$ of EnKF are evolved according to Equation (4) in~\cite{MR3654885}:
\begin{equation}
    \label{eq:ensemble_kalman_iteration}
    \vect u^{(j)}_{n+1} = \vect u^{(j)}_{n} + h\mat C^{up} (U_n) (h \mat C^{pp}(U_n) + \mat \Gamma)^{-1} \left(\vect y^{(j)}_{n+1} - \mathcal G(\vect u^{(j)}_n) \right), \qquad j = 1, \dots, J,
\end{equation}
where $\mat C^{uu}$ (used later),  $\mat C^{up}$ and $\mat C^{pp}$ are given by
\begin{align*}
    &  \mat C^{uu}(U) = \frac{1}{J} \sum_{j=1}^{J} (\vect u^{(j)} - \bar {\vect u}) \otimes (\vect u^{(j)} - \bar {\vect u}),
    \qquad \mat C^{up}(U) = \frac{1}{J} \sum_{j=1}^{J} (\vect u^{(j)} - \bar {\vect u}) \otimes (\mathcal G (\vect u^{(j)}) - \bar {\mathcal G}), \\
    & \mat C^{pp}(U) = \frac{1}{J} \sum_{j=1}^{J} (\mathcal G (\vect u^{(j)}) - \bar {\mathcal G}) \otimes (\mathcal G (\vect u^{(j)}) - \bar {\mathcal G}),
    \qquad \bar {\vect u} = \frac{1}{J} \sum_{j=1}^{J} \vect u^{(j)},
    \qquad \bar {\mathcal G} = \frac{1}{J} \sum_{j=1}^{J} \mathcal G(\vect u^{(j)}),
\end{align*}
and $\vect y^{(j)}_n = \vect y + \vect \eta^{(j)}_n$, where $\{\vect \eta_n^{(j)}\}$ are i.i.d.\ vectors with $\vect \eta_1^{(1)} \sim \mathcal N(0, h^{-1} \mat \Sigma)$.
Traditionally, the distribution of the noise employed to perturb the simulated observations $\{\mathcal G(\vect u_n^{(j)})\}$ in the EnKF
coincides with that of the noise in the observational model,
which suggests taking $\mat \Sigma = \mat \Gamma$.
It was shown in~\cite{MR3654885}, however, that taking $\mat \Sigma = 0$ also produces an efficient method for solving inverse problems.
Furthermore, the authors noticed that,
when taking the limit $h \to 0$,
\cref{eq:ensemble_kalman_iteration} is a tamed Euler--Maruyama-type discretization of the SDE
\begin{align}
    \label{eq:dynamical_system_enkf}%
    \dot {\vect u}^{(j)} = \frac{1}{J} \sum_{k=1}^{J} \eip*{\mathcal  G(\vect u^{(k)}) - \bar {\mathcal G}}{\vect y - \mathcal G(\vect u^{(j)}) + \sqrt{\mat \Sigma} \, \dot {\vect W}^{(j)}}[\mat \Gamma] (\vect u^{(k)} - \bar {\vect u}), \qquad j = 1, \dotsc, J,
\end{align}
where ${\vect W}^{(j)}$, $j = 1, \dots, J$, are standard independent Brownian motions.
They carried out a thorough analysis of this continuous-time dynamics in the particular case where the forward model $\mathcal G$ is linear and $\mat \Sigma = 0$.
\Cref{eq:dynamical_system_enkf} can be now viewed as a derivative-free approach to inverse problems,
which was recently referred in~\cite{2019arXiv190308866G} as the Ensemble Kalman Inversion (EKI) method.

More recently, in~\cite{2019arXiv190308866G},
a modification of \cref{eq:dynamical_system_enkf} with $\mat \Sigma=0$ was suggested to enable sampling from the posterior distribution over an infinite time horizon;
the modified dynamics read
\begin{equation}
    \label{eq:dynamical_system_enkf2}%
\dot {\vect u}^{(j)} = \frac{1}{J} \sum_{k=1}^{J} \eip*{\mathcal  G(\vect u^{(k)}) - \bar {\mathcal G}}{\vect y - \mathcal G(\vect u^{(j)})}[\mat \Gamma] (\vect u^{(k)} - \bar {\vect u}) - \mat C^{uu}(U) \mat \Gamma_0^{-1} \vect u^{(j)} + \sqrt{2 \mat C^{uu}(U)} \, \dot {\vect W}^{(j)},
\end{equation}
for $j = 1, \dotsc, J$.
The second term in the right hand side is included so as to take the prior information into account.
The idea of including the covariance matrix $\mat C^{uu}(U)$ in that term,
as well as in the noise,
is motivated by the fact,
in the case of linear forward model,
\cref{eq:dynamical_system_enkf2} can equivalently be written as
\begin{equation}
    \label{eq:dynamical_system_enkf_nusken}%
    \dot {\vect u}^{(j)} = - \mat C^{uu}(U) \, \grad \Phi_R(\vect u^{(j)}) + \sqrt{2 \mat C^{uu}(U)} \, \dot {\vect W}^{(j)}, \qquad j = 1, \dots, J,
\end{equation}
which is expected to produce approximate samples of the posterior of the inverse problem for large $J$.
Indeed, the formal mean field limit of this interacting particle system is given by the law of the process defined by the McKean-type SDE
\begin{equation}
    \label{eq:overdamped_langevin}%
    \dot {\vect u} = - \mathcal C(f_t) \, \grad \Phi_R(\vect u) + \sqrt{2 \mathcal  C(f_t)} \, \dot {\vect W}, \qquad f_t := \Law(\vect u_t),
\end{equation}
which clearly admits $\frac{1}{Z} \, \e^{-\Phi_R}$ as an invariant measure,
where $Z$ is the normalization constant.
The associated Fokker--Planck equation for $f$ is given by \eqref{eq:mean-field_equation} (with $\sigma = 1$);
it was derived formally in~\cite{2019arXiv190308866G} and rigorously in \cite{DL19}.
Two remarks are in order.
First, we note that a concentration of the particles at any point of $\real^d$ is a stationary solution of the dynamics~\eqref{eq:dynamical_system_enkf2}
and, likewise, any Dirac delta is a stationary solution of \cref{eq:overdamped_langevin}.
Second, as recently noted in~\cite{2019arXiv190810890N},
the $J$-particle distribution $ \left(\frac{1}{Z}\right)^J \prod_{j=1}^{J} \e^{-\Phi_R(\vect u^{(j)}; \vect y)}$
is not invariant under the dynamics \cref{eq:dynamical_system_enkf_nusken}.

The strategy of the proof of the stability estimates~\eqref{eq:main_result} is the following:
we first realize that the moments up to second order of the equation \eqref{eq:mean-field_equation} are governed by a closed system of ODEs.
This is a common feature appearing in some of the simplest cases of homogeneous kinetic equations,
such as
the Fokker-Planck operator preserving the first two moments of the distribution function \cite{T99},
the Maxwellian molecules case for the Landau--Fokker--Planck equation \cite{V98},
and the Boltzmann equation for Maxwellian molecules;
see \cite{MR2355628,CCC} and the references therein.
Then, we focus on finding stability estimates for solutions that have the same covariance matrix,
which is simpler because the nonlinearity of the problem does not show up and we are reduced to a kind of linear Fokker--Planck equation.
Then we obtain the stability estimate for any two solutions,
regardless of the values of their first two moments,
by using optimal transport techniques.
The strategy of our proofs follows that employed in similar results for the Boltzmann equation in the Maxwellian case as in \cite{BT05,BT06,BC07,MR2355628}.

The paper is organized as follows. In \cref{sec:preliminaries}, we summarize known results and
we present some equilibration estimates for the first and second moments of the solution to \cref{eq:mean-field_equation}.
In \cref{sec:stability_on_wasserstein}, we give a simple proof of the stability estimates~\eqref{eq:main_result} in euclidean Wassertein distance based on analytical techniques in optimal transport.

\section{Preliminaries}%
\label{sec:preliminaries}%

We remind the reader that the forward model $\mathcal G = G$ is assumed to be linear throughout the paper,
and we recall the following result, proved in~\cite{2019arXiv190308866G}.
\begin{proposition}
    [Closed system of ordinary differential for the first and second moments]
    \label{proposition:closed_system_moments}%
Assume $f_t$ is a solution of \cref{eq:mean-field_equation},
and let $\mat C(t) := \mathcal C(f_t)$ and $\vect \delta(t) := \mathcal M(f_t) - \vect u_0$,
where $\mathcal M(f_t)$ denotes the first moment of $f_t$.
The evolution of $\mat C(t)$ and $\vect \delta(t)$ is governed by the system:
\begin{subequations}
\begin{align}
    \label{eq:closed_equation_first_moment}%
    \dot {\vect \delta} (t) &= - \, \mat C(t) \, \mat B^{-1} \, \vect \delta(t),  \qquad &\left(\dot \dummy := \derivative*{1}{t}\dummy \right) \\
    \label{eq:closed_equation_second_moment}%
    \dot {\mat C} (t) &= - 2 \, \mat C(t) \, \mat B^{-1} \, \mat C(t) + 2 \sigma \mat C(t).
\end{align}
\end{subequations}
\end{proposition}
\begin{proof}
We show this only in the case $\sigma = 0$, for simplicity.
Multiplying \cref{eq:mean-field_equation} by $\vect u$,
integrating over $\real^d$,
and using the notation $\vect m(t) = \mathcal M(f_t)$, we obtain
\begin{align*}
    \dot {\vect m}(t)
    = - \, \mat C(t) \, \grad \Phi_R(\vect m(t), \vect y)
    &= - \, \mat C(t) \, \left(\mat G^T \mat \Gamma^{-1} (G \vect m(t) - \vect y ) + \mat \Gamma_0^{-1} \vect m(t) \right) \\
    &= - \, \mat C(t) \, \mat B^{-1} (\vect m(t) - \vect u_0),
\end{align*}
leading to \cref{eq:closed_equation_first_moment}.
Similarly, multiplying \cref{eq:mean-field_equation} by $\left(\vect u - \vect m(t)\right) \otimes \left(\vect u - \vect m(t)\right)$
and noticing that
\[
    \int_{\real^d} \derivative{1}{t} \big(\left(\vect u - \vect m(t)\right) \otimes \left(\vect u - \vect m(t)\right) f(\vect u, t) \big) \, \d \vect u = \int_{\real^d}  \left(\vect u - \vect m(t)\right) \otimes \left(\vect u - \vect m(t)\right) \derivative{1}[f]{t}(\vect u, t) \, \d \vect u,
\]
we obtain an equation for the covariance matrix.
Omitting the dependence of $\mat C$ and $\vect m$ on $t$ for convenience,
\begin{align*}
    \derivative*{1}{t} C_{ij} (t) &= - \int_{\real^d} \mat C : \left( \grad \big((u_i - m_i)(u_j - m_j)\big) \otimes \grad \Phi_R(\vect u, \vect y) \right) \, f(\vect u, t) \, \d \vect u, \\
                                  &= - \sum_{k,\ell} \int_{\real^d} C_{k \ell} \big( \delta_{ki} (u_j - m_j) + \delta_{kj} (u_i - m_i) \big) \, (\mat B^{-1}(\vect u - \vect u_0))_\ell \, f(\vect u, t) \, \d \vect u.
\end{align*}
Since the term in the first round brackets in the integral is mean-zero with respect to $f(\vect u, t)$,
we can remove and add constants in the other factor:
\begin{align*}
    \derivative*{1}{t} C_{ij} (t)  &= - \sum_{k,\ell} \int_{\real^d} C_{k \ell} \big( \delta_{ki} (u_j - m_j) + \delta_{kj} (u_i - m_i) \big) \, (\mat B^{-1} (\vect u - \vect m))_\ell \, f(\vect u, t) \, \d \vect u, \\
                                   &= - \sum_{\ell, p} \int_{\real^d} \big( C_{i \ell} \, (u_j - m_j) + C_{j \ell} \, (u_i - m_i) \big) \, B^{-1}_{\ell p} \, (u_p - m_p) \, f(\vect u, t) \, \d \vect u, \\
                                   &= - \sum_{\ell, p} \big( C_{i \ell} \, C_{j p} + C_{j \ell} \, C_{i p} \big) \, B^{-1}_{\ell p} = - 2 \sum_{\ell, p} C_{i \ell} \, C_{j p} \, B^{-1}_{\ell p},
\end{align*}
which, in matrix form, gives \cref{eq:closed_equation_second_moment}.
\end{proof}

If we assume that $\mat C_0 := \mat C(f_0)$ is positive definite,
then the solution of \cref{eq:closed_equation_second_moment} reads
\begin{equation}
    \label{eq:expression_covariance_matrix}%
    \mat C(t) =
    \begin{cases}
        \left( \frac{1 - \e^{-2 \sigma t}}{\sigma} \, \mat B^{-1} + \e^{- 2\sigma t} \, \mat C_0^{-1}\right)^{-1}
        \qquad &\text{if } \sigma > 0, \\
        \left(2 \mat B^{-1} t + \mat C_0^{-1}\right)^{-1}
        \qquad &\text{if } \sigma = 0.
    \end{cases}
\end{equation}
We notice that the solution in the case $\sigma = 0$ is the pointwise limit as $\sigma \to 0$ of that when $\sigma > 0$.
For a given solution $\mat C(t)$ of \cref{eq:closed_equation_second_moment},
we will denote by $\mat U(s, t; \mat C)$ the fundamental matrix associated with \cref{eq:closed_equation_first_moment};
this matrix solves
\begin{equation}
    \label{eq:fundamental_solution_definition_U}%
    \forall s \in \real, t \geq s: \qquad \partial_t \mat U(s, t; \mat C) = - \mat C(t) \, \mat B^{-1} \, \mat U(s, t; \mat C), \qquad \mat U(s, s; \mat C) = \mat I.
\end{equation}
\begin{lemma}
    [Bound for the fundamental matrix]
    \label{lemma:bound_for_the_fundamental_matrix}%
    Let $\mat C(t)$ be a solution of \cref{eq:closed_equation_second_moment} with initial condition $\mat C(0)$.
    The matrix $\mat U(s, t) := \mat U(s, t; \mat C)$ satisfies
    \begin{equation}
        \label{eq:bound_covariance_u}%
        \eucnorm{\mat U(s, t)}[2] \leq \e^{-\sigma (t - s)} \, \sqrt{\frac{\alpha(s)}{\alpha(t)}}
        \, \sqrt{\max(\eucnorm{\mat C(0)}[2], \eucnorm{\mat B}[2])}
        \, \sqrt{\max(\eucnorm{\mat C(0)^{-1}}[2], \eucnorm{\mat B^{-1}}[2])},
    \end{equation}
    where
    \begin{equation}
        \label{eq:definition_alpha}%
        \alpha(t) =
        \begin{cases}
            2 t + 1 & \text{if } \sigma = 0, \\
            \frac{1}{\sigma} (1-\e^{-2 \sigma t}) + \e^{-2 \sigma t} & \text{if } \sigma > 0.
        \end{cases}
    \end{equation}
    \end{lemma}
\begin{proof}
We notice that
\[
    \derivative*{1}{t} (\mat U(s, t)^T \, \mat C(t)^{-1} \, \mat U(s, t)) = - 2 \sigma (\mat U(s, t)^T \, \mat C(t)^{-1} \, \mat U(s, t)),
\]
which implies
\begin{equation}
    \label{eq:equation_convenient_u}%
    \Big(\mat C(t)^{-1/2} \mat U(s, t)\Big)^T \Big(\mat C(t)^{-1/2} \mat U(s, t)\Big) =  \mat U(s, t)^T \, \mat C(t)^{-1} \, \mat U(s, t) = \e^{-2 \sigma (t-s)} \mat C(s)^{-1}.
\end{equation}
Let us denote the polar decomposition of $\mat C(t)^{-1/2} \mat U(s, t)$ by $\mat Q(s, t) \, \mat S(s, t)$,
for some orthogonal matrix $\mat Q(s, t)$ and some symmetric matrix $\mat S(s, t)$.
Substituting this decomposition in \cref{eq:equation_convenient_u},
we obtain $\mat S(s, t) = \e^{- \sigma (t-s)} \mat C(s)^{-1/2}$
and so $\mat U(s, t) = \e^{-\sigma (t-s)} \mat C(t)^{1/2} \, \mat Q(s, t) \, \mat C(s)^{-1/2}$.
In particular,
\begin{align*}
    \eucnorm{\mat U(s, t)}[2] \leq \e^{-\sigma (t - s)} \, \sqrt{\eucnorm{\mat C(t)}[2] \eucnorm{\mat C(s)^{-1}}[2] }.
\end{align*}
Rewriting $\mat C(t)$ in a way that exhibits a convex combinations of $\mat B^{-1}$ and $\mat C(0)^{-1}$,
\begin{equation}
    \label{eq:convex_combination}%
    \mat C(t) = \frac{1}{\alpha(t)} \left( (1-\beta(t)) \, \mat B^{-1} + \beta(t) \, \mat C(0)^{-1} \right)^{-1},
     \qquad \beta(t) = \frac{\e^{-2 \sigma t}}{\alpha(t)},
\end{equation}
we deduce \cref{eq:bound_covariance_u}.
\end{proof}

In the sequel,
$\alpha(t)$ denotes the same function as in \cref{lemma:bound_for_the_fundamental_matrix},
and we employ the notations $\eucnorm{\dummy}[F] := \sum_{ij} \dummy_{ij}^2$ and $\eucnorm{\dummy}[2]$ to denote
the Frobenius matrix norm and the operator norm induced by the euclidean vector norm in $\real^d$, respectively.

\begin{lemma}
    [Convergence of the first and second moments]
    \label{lemma:moments_convergence_noise}%
    We consider two solutions $\mat C_1(t)$, $\mat C_2(t)$ of \cref{eq:closed_equation_second_moment}
    and the corresponding solutions $\vect \delta_1(t)$, $\vect \delta_2(t)$ of \cref{eq:closed_equation_first_moment},
    and we assume that
    \begin{align*}
        &\eucnorm{\mat C_1(0)}[2] \vee \eucnorm{\mat C_2(0)}[2] \vee \eucnorm{\mat B}[2] \leq M, \\
        &\eucnorm{\mat C_1(0)^{-1}}[2] \vee \eucnorm{\mat C_2(0)^{-1}}[2] \vee \eucnorm{\mat B^{-1}}[2] \leq m, \\
        &\eucnorm{\vect \delta_1(0)}[2] \vee \eucnorm{\vect \delta_1(0)}[2] \leq R.
    \end{align*}
    Then it holds that
    \begin{subequations}
    \begin{align}
        \label{eq:decrease_covariance}%
        & \eucnorm{\mat C_1(t) - \mat C_2(t)}[F]  \leq M^2 \, m^2 \, \eucnorm{\mat C_1(0) - \mat C_2(0)}[F] \,  \frac{\e^{-2 \sigma t}}{\alpha(t)^2}, \\
        \label{eq:decrease_first_moment}%
        & \eucnorm{\vect \delta_1(t) - \vect \delta_2(t)}[2] \leq \left( \sqrt{mM} \, \eucnorm{\vect \delta_1(0) - \vect \delta_2(0)}[2]  + \frac{1}{2} \, m^4 M^3 R \, \eucnorm{\mat C_2(0) - \mat C_1(0)}[F] \right) \frac{\e^{- \sigma t}}{\sqrt{\alpha(t)}},
    \end{align}
    \end{subequations}
\end{lemma}
\begin{proof}
    By a sub-multiplicative property of the Frobenius norm,
    \[
        \eucnorm{\mat C_1(t) - \mat C_2(t)}[F] = \eucnorm{\mat C_1(t)}[2] \, \eucnorm{\mat C_1(t)^{-1} - \mat C_2(t)^{-1}}[F] \, \eucnorm{\mat C_2(t)}[2],
    \]
    We observe $\mat C_1(t)^{-1} - \mat C_2(t)^{-1} = \e^{-2 \sigma t} (\mat C_1(0)^{-1} - \mat C_2(0)^{-1})$ so,
    using the sub-multiplicative property of the norm again,
    \begin{equation}
        \label{eq:intermediate_equation_moment_bound_covariance}%
        \eucnorm{\mat C_1(t) - \mat C_2(t)}[F] = \eucnorm{\mat C_1(t)}[2] \eucnorm{\mat C_1(0)^{-1}}[2] \, \eucnorm{\mat C_0(t) - \mat C_2(0)}[F] \, \eucnorm{\mat C_2(0)^{-1}}[2] \, \eucnorm{\mat C_2(t)}[2] \, \e^{-2 \sigma t}.
    \end{equation}
    Since $\frac{1}{\alpha(t)}\mat C_i(t)^{-1}$ is a convex combination of $\mat C_i(0)^{-1}$ and $\mat B^{-1}$,
    \begin{align*}
        \eucnorm{\mat C_i(t)}[2] \leq \frac{1}{\alpha(t)} \, \max \Big( \eucnorm{\mat C_i(0)}[2], \eucnorm{\mat B}[2] \Big), \qquad i = 1, 2,
    \end{align*}
    leading to \cref{eq:decrease_covariance}.

    For the first moments,
    we have
    \begin{align*}
        \derivative*{1}{t} (\vect \delta_1(t) - \vect \delta_2(t) ) = - \mat C_1(t) \, \mat B^{-1} (\vect \delta_1(t) - \vect \delta_2(t)) - (\mat C_2(t) - \mat C_1(t)) \, \mat B^{-1} \, \vect \delta_2(t).
    \end{align*}
    By the variation-of-constants formula,
    and with the shorthand notation $\mat U_i(s, t) := \mat U(s, t, \mat C_i)$,
    we deduce that
    \begin{align*}
        \vect \delta_1(t) - \vect \delta_2(t) = - \mat U_1(s, t) (\vect \delta_1(s) - \vect \delta_2(s)) - \int_{s}^{t} \mat U_1(u, t) (\mat C_2(u) - \mat C_1(u)) \, \mat B^{-1} \, \vect \delta_2(u) \, \d u.
    \end{align*}
    Employing \cref{eq:bound_covariance_u,eq:decrease_covariance},
    and using the fact that $\vect \delta_2(\vect u) = \mat U_2(s, u) \, \vect \delta_2(s)$,
    we obtain
    \begin{align}
        \label{eq:intermediate_step_moment_bound}%
        \eucnorm{\vect \delta_1(t) - \vect \delta_2(t)}[2] & \leq  \sqrt{mM} \, \sqrt{\frac{\alpha(s)}{\alpha(t)}} \, \e^{- \sigma (t - s)} \, \eucnorm{\vect \delta_1(s) - \vect \delta_2(s)}[2] \\
        \notag%
        & \quad + m^3 M^3 \eucnorm{\vect \delta_2(s)} \, \sqrt{\frac{\alpha(s)}{\alpha(t)}} \, \e^{- \sigma(t - s)} \, \eucnorm{\mat C_2(0) - \mat C_1(0)}[F]
        \int_{s}^{t} \frac{\e^{-2 \sigma u}}{\alpha(u)^2} \eucnorm{\mat B^{-1}}[2] \, \vect \, \d u.
    \end{align}
    We calculate that :
    \begin{align}
        \notag%
        I(s, t) &:= \int_{s}^{t} \frac{\e^{-2 \sigma u}}{\alpha(u)^2} \, \d u
        = \frac{1}{2 (\sigma - 1)} \left(\frac{1}{\alpha(t)} - \frac{1}{\alpha(s)}\right)
        = \frac{\e^{-2 \sigma s} - \e^{-2 \sigma t}}{2 \, \sigma \, \alpha(s) \, \alpha(t)} , \qquad \sigma \neq 0, 1, \\
        \label{eq:integral}%
        &\leq \lim_{t \to \infty} I(s, t) = \frac{\e^{-2 \sigma s} - \e^{-2 \sigma \infty}}{2 \, \sigma \, \alpha(s) \, \alpha(\infty)} = \frac{\e^{-2 \sigma s}}{2 \, \alpha(s)}.
    \end{align}
    (This calculation fails for $\sigma = 0$ and $\sigma = 1$,
    but it is easy to check that the conclusion holds for any $\sigma \geq 0$.)
    This leads to \cref{eq:decrease_first_moment} after taking $s = 0$ (the case $s > 0$ will be useful in \cref{remark:convergence_of_the_fundamental_matrices} below) and rearranging.
\end{proof}

We note that, in the case $\sigma = 0$,
\cref{eq:decrease_covariance} cannot be employed,
by letting $\mat C_2(0) \to 0$,
to deduce the rate of convergence of $\mat C_1(t)$ to 0,
because the bound $m$ in the assumptions grows to $+ \infty$ as $\mat C_2(0) \to 0$.
It can, however, be employed (setting $\vect \delta_2(0) = 0$ and $\mat C_2(0) = \mat C_1(0)$) to deduce that
$\vect \delta_1(t)$ converges to zero with rate $\e^{- \sigma t}/\sqrt{\alpha(t)}$,
which is consistent with \cref{eq:bound_covariance_u}.

\begin{remark}
    \label{remark:convergence_of_the_fundamental_matrices}
    Since $\vect \delta_i(t) =  \mat U_i(s, t) \, \vect \delta_i(s)$, for $i = 1, 2$, by definition of $\mat U_i(s, t)$,
    it follows from \cref{eq:intermediate_step_moment_bound} that
    \begin{equation}
        \label{eq:contraction_U}%
        \forall s \leq t, \qquad \eucnorm{\mat U_2(s, t) - \mat U_1(s, t)}[2] \leq m^4 M^3 \, \eucnorm{\mat C_2(0) - \mat C_1(0)}[F] \, \frac{\e^{-\sigma (s + t)}}{\sqrt{\alpha(s) \alpha(t)}},
    \end{equation}
    where the constants $m$ and $M$ are defined as before.
\end{remark}

In the rest of this paper, we denote by $g(\dummy; \vect \mu, \mat \Sigma)$ the density of the Gaussian $\mathcal N(\vect \mu, \mat \Sigma)$.
\begin{lemma}
    [Propagation of Gaussians for the linear equation]
    \label{lemma:fundamental_solution}%
    Let $\mat C(t)$ be the solution of
    \[
        \dot {\mat C} (t) = - 2 \, \mat C(t) \, \mat B^{-1} \, \mat C(t) + 2 \sigma \mat C(t), \qquad \mat C(0) = \mat C_0,
    \]
    for a given matrix $\mat C_0$.
    Then the solution of the linear Fokker--Planck equation
    \begin{subequations}
    \begin{align}
        \label{eq:mean-field_equation_linear}%
        &\derivative{1}[f]{t}(\vect u, t) = \grad \cdot \left(\mat C(t) \mat B^{-1}(\vect u - \vect u_0) \, f(\vect u, t) \right)  + \sigma \grad \cdot (\mat C(t) \, \grad f(\vect u, t)), \\
        &f(\vect u, 0) = g(\vect u; \vect \mu_0, \mat \Sigma_0),
    \end{align}
    \end{subequations}
    is given by the Gaussian density $f(\vect u, t) = g(\vect u; \vect \mu(t), \mat \Sigma(t))$
    where
    \begin{subequations}
    \begin{align}
        \label{eq:fundamental_solution_first_moment}%
        \vect \mu(t) &= \vect u_0 + \mat U(0, t) \, (\vect \mu_0 - \vect u_0), \\
        \label{eq:fundamental_solution_second_moment}%
        \mat \Sigma(t) &= \mat U(0,t) \, \mat \Sigma_0 \mat U(0, t)^T + 2 \sigma \, \int_{0}^{t} \mat U(s, t) \, \mat C(s) \, \mat U(s, t)^T \, \d s.
    \end{align}
    \end{subequations}
    Here $\mat U(\dummy, \dummy) := \mat U(\dummy, \dummy; \mat C)$ is given by \cref{eq:fundamental_solution_definition_U}.
    If $\mat \Sigma_0 = 0$,
    then the matrix $\mat \Sigma(t)$ admits the following explicit expression in terms of $\mat C(t)$:
    \begin{equation}
        \label{eq:expression_sigma}
        \mat \Sigma(t) = (1 - \e^{-2 \sigma t}) \, \mat C(t).
    \end{equation}
\end{lemma}
\begin{proof}
Proceeding as in \cref{proposition:closed_system_moments},
we deduce that the first and second moments of any solution to \cref{eq:mean-field_equation_linear},
which we denote $\vect \mu$ and $\mat \Sigma$, satisfy
\begin{subequations}
\label{eq:system_moments_linear}%
\begin{align}
    \label{eq:equation_mu}%
    \dot {\vect \mu}(t) &= - \mat C(t) \, \mat B^{-1}(\vect \mu(t) - \vect u_0), \\
    \label{eq:equation_sigma}%
    \dot {\mat \Sigma}(t) &= - \mat C(t) \, \mat B^{-1} \, \mat \Sigma(t) - \mat \Sigma(t) \, \mat B^{-1} \, \mat C(t) + 2 \sigma \mat C(t).
\end{align}
\end{subequations}
We then verify,
proceeding similarly to~\cite{duongmean,2016arXiv161101593A},
that the Gaussian ansatz
\[
    f(\vect u, t) = \frac{1}{(2\pi)^{d/2} \, \sqrt{\det \mat \Sigma(t)}} \, \e^{ - \frac{1}{2} \, (\vect u - \vect \mu(t))^T \mat \Sigma(t)^{-1} (\vect u - \vect \mu(t)) }
\]
is indeed a solution.
Omitting the dependence of $\mat C$, $\vect \mu$ and $\mat \Sigma$ on $t$ for notational convenience,
we calculate that the left-hand side of \cref{eq:mean-field_equation_linear} reads
\begin{align*}
    \frac{LHS}{f(\vect u, t)} =
        \dot {\vect \mu}^T \mat \Sigma^{-1} (\vect u - \vect \mu)
        + \frac{1}{2} \, (\vect u - \vect \mu)^T {\mat \Sigma}^{-1} \, \dot {\mat \Sigma} \, {\mat \Sigma}^{-1} (\vect u - \vect \mu)
        + \frac{1}{2 \det \mat \Sigma} \derivative*{1}{t} \left( \det \mat \Sigma\right)
\end{align*}
and the right hand is
\begin{align*}
    \frac{RHS}{f(\vect u, t)} =
        - (\vect u - \vect u_0)^T \mat B^{-1} \mat C \mat \Sigma^{-1} (\vect u - \vect \mu)
        + \sigma (\vect u - \vect \mu) \mat \Sigma^{-1} \mat C \mat \Sigma^{-1} (\vect u - \vect \mu)
        + \trace(\mat C \mat B^{-1} - \sigma \mat C \mat \Sigma^{-1} ).
\end{align*}
Both sides of the equation are quadratic polynomials in $\vect u$.
Equating the Hessians w.r.t. $\vect u$ of the coefficients of both sides,
and multiplying left and right by $\mat \Sigma$, we obtain \cref{eq:equation_sigma}.
Taking this equation into account and equating the gradients,
we obtain \cref{eq:equation_mu}.
It remains to check that the constant terms (w.r.t. $\vect u$) coincide,
which can be seen from
the fact that $\trace(\mat C \mat B^{-1} - \sigma \mat C \mat \Sigma^{-1} ) = - \frac{1}{2} \trace(\dot {\mat \Sigma} \mat \Sigma^{-1})$, by \cref{eq:equation_sigma},
and the formula for the derivative of the determinant function:
\[
    0 = \derivative*{1}{t} \int_{\real^{d}} \frac{1}{(2\pi)^{d/2} \, \sqrt{\det \mat \Sigma}} \, \e^{ - \frac{1}{2} \, \vect u^T \mat \Sigma^{-1} \vect u } \, \d \vect u
    = \frac{1}{2}\left( \trace (\dot {\mat \Sigma} \mat \Sigma^{-1}) - \frac{1}{\det \mat \Sigma} \derivative*{1}{t} \left( \det \mat \Sigma\right) \right).
\]
For general initial conditions $\vect \mu_0$ and $\mat \Sigma_0$,
we can check that the solution to the system of equations~\eqref{eq:system_moments_linear} is given by \cref{eq:fundamental_solution_first_moment,eq:fundamental_solution_second_moment}.
\Cref{eq:expression_sigma} can be checked by simple substitution in~\cref{eq:equation_sigma}.
\end{proof}

\begin{remark}
We remark that, for $\sigma > 0$,
$\mat U(t - s, t) \to \e^{- \sigma s} \, \mat I$ as $t \to \infty$ because $\mat C(t) \to \sigma \mat B$.
Therefore, employing \cref{lemma:bound_for_the_fundamental_matrix,lemma:fundamental_solution} and dominated convergence, we deduce
\begin{align*}
    \lim_{t \to \infty} \mat \Sigma(t) &= 2 \sigma \lim_{t \to \infty} \int_{0}^{t} \mat U(s, t) \, \mat C(s) \, \mat U(s, t)^T \, \d s, \\
                                  &= 2 \sigma  \lim_{t \to \infty} \int_{0}^{t} \mat U(t - s, t) \, \mat C(t - s) \, \mat U(t - s, t)^T \, \d s, \\
                                  &= 2 \sigma \lim_{t \to \infty} \int_{0}^{\infty} \mat U(t - s, t) \, \mat C(t - s) \, \mat U(t - s, t)^T \, \one_{\{\dummy \leq t\}}(s) \,  \d s, \\
                                  &= 2 \sigma  \lim_{t \to \infty} \int_{0}^{\infty} \e^{- \sigma s} \, \sigma \, \mat B \, \e^{- \sigma s}  \d s = \sigma \mat B \left(= \lim_{t \to \infty}  \mat C(t)\right),
\end{align*}
which holds for any initial condition $\mat \Sigma_0$.
\end{remark}

\begin{remark}
    A byproduct of \cref{lemma:fundamental_solution} is that
    the mean field equation~\eqref{eq:mean-field_equation} too propagates Gaussians
    when the forward model $G$ is linear,
    as was proved in~\cite[Proposition 4]{2019arXiv190308866G}.
\end{remark}

\section{Stability in Wasserstein distance}%
\label{sec:stability_on_wasserstein}

The aim of this section is to derive a stability property for both the linear Fokker--Planck equation~\eqref{eq:mean-field_equation_linear} (where $\mat C(t)$ is a given parameter)
and the nonlocal mean field equation~\eqref{eq:mean-field_equation} (where $\mathcal C(f_t)$ depends on the solution),
which we undertake in \cref{sub:stability_for_the_linear_fokker_planck_equation} and \cref{sub:stability_for_general_solutions}, respectively.

\subsection{Stability for the linear Fokker--Planck equation \texorpdfstring{\cref{eq:mean-field_equation_linear}}{}}%
\label{sub:stability_for_the_linear_fokker_planck_equation}

Throughout this subsection we consider that $\mat C(t)$ is a given solution of \cref{eq:closed_equation_second_moment} and $\mat U(0, t) = \mat U(0, t; \mat C)$.
For some probability measure $f$ over $\real^d$ and a mapping $T: \real^d \to \real^d$,
we will denote the pushforward measure by $T_\sharp f$.
We remind the reader that, if $f$ admits a density $\hat f$ with respect to the Lebesgue measure
and $\mat A \in \real^{d\times d}$ is a nonsingular matrix,
then $\mat A_{\sharp}f$ (identifying $\mat A$ with the associated linear mapping) has density $\frac{1}{\det(\mat A)} \, \hat f(\mat A^{-1} \dummy)$.
We show the following result.
\begin{proposition}
    [Convergence of solutions when the covariance is given]
    \label{proposition:convergence_over_manifold_same_covariance}%
    Let $f^1$ and $f^2$ be two solutions of \cref{eq:mean-field_equation_linear}
    associated with initial conditions $f^1_0$ and $f^2_0$, respectively.
    Then
    \begin{align}
        \label{eq:stability_same_manifold}%
        W_2(f^1_t, f^2_t) &\leq \sqrt{\eucnorm{\mat U(0, t)^T \mat U(0, t)}[2]} \, W_2(f_0^1, f_0^2),
    \end{align}
    Under the same assumptions as in \cref{lemma:moments_convergence_noise},
    it therefore holds, in view of \cref{eq:bound_covariance_u},
    \begin{align*}
        W_2(f^1_t, f^2_t) \leq \sqrt{m M} \, \left(\frac{\e^{-\sigma t}}{\sqrt{\alpha(t)}}\right) \, W_2(f_0^1, f_0^2).
    \end{align*}
\end{proposition}

To prove \cref{proposition:convergence_over_manifold_same_covariance} we will need the following lemma.
\begin{lemma}
    [Influence of linear transformations on the Wasserstein distance]
    \label{lemma:wasserstein_linear}%
    Let $\mat A \in \real^{d \times d}$ be nonsingular
    and let us consider two probability measures with finite second moment, $f, g \in \mathcal P_2(\real^d)$.
    Then also $\mat A_\sharp f, \mat A_\sharp g \in \mathcal P_2(\real^d)$ and
    \[
        W_2\left(\mat A_\sharp f, \mat A_\sharp g\right) \leq \sqrt{\eucnorm{\mat A^T \mat A}[2]} W_2(f, g).
    \]
\end{lemma}
\begin{proof}
    Let $\gamma_{o}$ be an optimal transference plan
    (by~\cite[Proposition 2.1]{MR2355628}, the infimum in the definition of the Wasserstein distance is achieved)
    such that
    \[
        \int \!\!\! \int_{\real^{d} \times \real^d} \eucnorm{x - y}^2 \, \gamma_{o}(\d x \, \d y) = {W_2(f, g)}^2,
    \]
    and consider the map $r: (x, y) \mapsto (\mat A x, \mat A y)$.
    The pushforward plan $r_\sharp \gamma_{o}$ has the correct marginals:
    looking for example at the $x$ marginal, we calculate that for all $\varphi \in C_b(\real^d)$,
    \begin{align*}
        \int \!\!\! \int_{\real^{d} \times \real^d} \varphi(x) \, r_\sharp \gamma_{o} (\d x \, \d y)
        &= \int \!\!\! \int_{\real^{d} \times \real^d} \varphi(\mat A x) \, \gamma_{o} (\d x \, \d y) \\ 
        &= \int_{\real^{d}} \varphi(\mat A x) \, f(\d x) 
        = \int_{\real^d} \varphi(x) \, \mat A_\sharp f(\d x). 
    \end{align*}
    Furthermore, by a change of variable,
    \begin{align*}
        \int \!\!\! \int_{\real^{d} \times \real^d} \eucnorm{x - y}^2 \, r_\sharp \gamma_{o}(\d x \, \d y) & = \int \!\!\! \int_{\real^{d} \times \real^d} \eucnorm{\mat A x - \mat A y}^2 \, \gamma_{o}(\d x \, \d y) \\
                                                                                                             & \leq \eucnorm{\mat A^T \mat A}[2] \int \!\!\! \int_{\real^{d} \times \real^d} \eucnorm{x - y}^2 \, \gamma_{o}(\d x \, \d y) = \eucnorm{\mat A^T \mat A}[2] \, {W_2(f, g)}^2.
    \end{align*}
    We notice that orthogonal transformations do not influence the Wasserstein distance.
\end{proof}

\begin{proof}[Proof of \cref{proposition:convergence_over_manifold_same_covariance}]
Let us denote by $\zeta(\vect u, t; \vect v)$ the fundamental solution provided by \cref{lemma:fundamental_solution}.
By linearity, the solution of \cref{eq:mean-field_equation_linear} associated with initial condition $f_0$ can be expressed as follows:
\begin{align*}
    f(\vect u, t) &= \int_{\real^{d}} f_0(\vect v) \, \zeta(\vect u, t; \vect v) \, \d \vect v \\
            &= \int_{\real^{d}} f_0(\vect v) \, g\left(\vect u; \vect u_0 + \mat U(0, t)(\vect v - \vect u_0), \mat \Sigma(t) \right) \, \d \vect v.
\end{align*}
By the change of variables $\vect v \mapsto \mat U(0, t) (\vect v - \vect u_0) =: \vect w(\vect v)$,
we can rewrite this integral as
\begin{align}
    \label{eq:intermediate_expression_convolution}%
    f(\vect u, t) &= \int_{\real^{d}} \frac{f_0\left(\mat U(0, t)^{-1}\vect w + \vect u_0\right)}{\det (\mat U(0, t))} \, g(\vect u; \vect u_0 + \vect w, \mat \Sigma(t)) \, \d \vect w \\
    \notag%
                  &= \left( \frac{f_0\left(\mat U(0, t)^{-1} \dummy + \vect u_0\right)}{\det (\mat U(0, t))} \star g(\, \dummy \,; \vect u_0, \mat \Sigma(t)) \right) (\vect u).
\end{align}
By the convexity property of the Wasserstein distance~\cite[Proposition 2.1]{MR2355628},
its invariance under translation
and \cref{lemma:wasserstein_linear},
we obtain
\begin{align*}
    W_2(f^1_t, f^2_t)
    &\leq W_2\left( \frac{f_0^1\left(\mat U(0,t)^{-1}\vect w + \vect u_0\right)}{\det (\mat U(0, t))}, \frac{f_0^2\left(\mat U(0,t)^{-1}\vect w + \vect u_0\right)}{\det (\mat U(0, t))}\right) \\
    &\leq \sqrt{\eucnorm{\mat U(0,t)^T \mat U(0,t)}[2]} \, W_2(f_0^1, f_0^2),
\end{align*}
which is the desired inequality.
\end{proof}

\begin{remark}
\Cref{proposition:convergence_over_manifold_same_covariance} can also be proved via a purely probabilistic approach,
employing the approach presented e.g. in~\cite{MR2964689,MR2459454}.
Indeed a solution of \cref{eq:mean-field_equation_linear} with initial condition $f_0$ can be viewed,
by It\^o's formula,
as the law of the process $(X_t)_{t\geq0}$ that solves the stochastic differential equation (SDE)
\[
    \d \vect X_t = - \mat C(t) \, \mat B^{-1} \, (\vect X_t - \vect u_0) \, \d t + \sqrt{2 \, \sigma \, \mat C(t)} \, \d W_t, \qquad \vect X_0 \sim f_0,
\]
where $W$ is a standard Wiener process on $\real^d$.
Considering two solutions $\vect X_t$ and $\vect Y_t$ associated with
the initial conditions $\vect X_0 \sim f^1_0$ and $\vect Y_0 \sim f^2_0$ (and with the same Wiener process),
we calculate
\[
    \d \vect X_t - \d \vect Y_t = - \mat C(t) \mat B^{-1} (\mat X_t - \mat Y_t) \, \d t,
\]
and therefore
\[
    \vect X_t - \vect Y_t = \mat U(0, t) (\vect X_0 - \vect Y_0),
\]
which implies
\begin{equation}
\label{eq:intermediate_equation_proba_approach}%
\eucnorm{\vect X_t - \vect Y_t}^2 \leq \eucnorm{\mat U(0, t)^T \mat U(0, t)}[2] \, \eucnorm{\vect X_0 - \vect Y_0}^2.
\end{equation}
Recalling that the Wasserstein distance can equivalently be defined as
\[
    W_2(\rho_1, \rho_2) = \left(\inf_{\vect X, \vect Y}  \expect \eucnorm{\vect X - \vect Y}^2 \right)^{1/2},
\]
where the infimum is over all $\vect X$ and $\vect Y$ with laws $\rho_1$ and $\rho_2$, respectively,
and taking the expectation of both sides of \cref{eq:intermediate_equation_proba_approach},
we obtain
\[
    W_2(f^1_t, f^2_t) \leq \eucnorm{\mat U(0, t)^T \mat U(0, t)}[2] \, \expect \eucnorm{\vect X_0 - \vect Y_0}^2.
\]
Infimizing over all $\vect X_0$ and $\vect Y_0$ with laws $f^1_0$ and $f^2_0$, respectively,
we obtain precisely \cref{eq:stability_same_manifold}.
\end{remark}
We remark that the first moment of $f^1$ and $f^2$ need not coincide for \cref{proposition:convergence_over_manifold_same_covariance} to hold.

\subsection{Stability for the nonlocal mean field equation}%
\label{sub:stability_for_general_solutions}
To prepare the terrain for the derivation of our main result,
we begin by showing a stability property on the set of Gaussian solutions.
To this end, we will employ the following bound for the distance between the square root of the covariant matrices associated with two solutions.
\begin{lemma}
    [Convergence of the square root of the covariance matrix]
    \label{lemma:moments_square_root}%
    Under the assumptions of \cref{lemma:moments_convergence_noise},
    it holds that
    \begin{align}
        \label{eq:decrease_covariance_square_root}%
        & \eucnorm{\mat C_1(t)^{1/2} - \mat C_2(t)^{1/2}}[F]  \leq C_R M \, m \, \eucnorm{\mat C_1(0)^{1/2} - \mat C_2(0)^{1/2}}[F]
        \frac{\e^{-\sigma t}}{\alpha(t)}
    \end{align}
    where $C_R$ is a constant that depends only on the dimension of the problem.
\end{lemma}
\begin{proof}
    We restrict ourselves in the proof to the case $\sigma > 0$ for simplicity.
    Employing the same reasoning as in the first part of the proof of \cref{lemma:moments_convergence_noise},
    we write
    \[
        \eucnorm{\mat C_1(t)^{1/2} - \mat C_2(t)^{1/2}}[F] = \eucnorm{\mat C_1(t)^{1/2}}[2] \, \eucnorm{\mat C_1(t)^{-1/2} - \mat C_2(t)^{-1/2}}[F] \, \eucnorm{\mat C_2(t)^{1/2}}[2],
    \]
    The middle term can be written as
    \[
        \eucnorm{\mat C_1(t)^{-1/2} - \mat C_2(t)^{-1/2}}[F] = \eucnorm{(\mat M + \mat M_1)^{1/2} - (\mat M + \mat M_2)^{1/2}}[F],
    \]
    where $\mat M = (1 - \e^{-2 \sigma t}) \mat B^{-1}/\sigma$ and $\mat M_i = \e^{-2 \sigma t} \mat C_i(0)^{-1}$, for $i = 1, 2$.
    Therefore, using the technical bound presented in \cref{lemma:technical_result_concavity_matrix_square_root} below,
    \begin{align*}
        \eucnorm{\mat C_1(t)^{-1/2} - \mat C_2(t)^{-1/2}}[F] \leq C_R \eucnorm{\mat M_1^{1/2} - \mat M_2^{1/2}}[F] &= C_R \, \e^{-\sigma t} \eucnorm{\mat C_1(0)^{-1/2} - \mat C_2(0)^{-1/2}}[F] \\
        &\leq C_R \, m \, \e^{-\sigma t} \,\eucnorm{\mat C_1(0)^{1/2} - \mat C_2(0)^{1/2}}[F],
    \end{align*}
    which leads to \cref{eq:decrease_covariance_square_root} after employing the convex decomposition~\eqref{eq:convex_combination} to bound $\eucnorm*{\mat C_i^{1/2}}[2]$.
\end{proof}

The Wasserstein distance between two Gaussian measures admits an explicit expression,
which we recall in the following lemma.
\begin{lemma}
    [Wasserstein distance between Gaussians]
    \label{lemma:wasserstein_distance_between_gaussians}%
    Consider two Gaussians probability measures $\mathcal N(\vect \mu_1, \mat \Sigma_1)$ and $\mathcal N(\vect \mu_2, \mat \Sigma_2)$ on $\real^d$.
    The Wasserstein distance between them is given by
    \begin{equation}
        \label{eq:exact_distance_gaussians}%
        \abs{W_2\left(\mathcal N(\vect \mu_1, \mat \Sigma_2), \mathcal N(\vect \mu_2, \mat \Sigma_2)\right)}^2
        = \eucnorm{\vect \mu_1 - \vect \mu_2}^2 + \trace \left(\mat \Sigma_1 + \mat \Sigma_2 - 2(\mat \Sigma_1^{1/2} \mat \Sigma_2 \mat \Sigma_1^{1/2})^{1/2}\right).
    \end{equation}
\end{lemma}
\begin{proof}
    \Cref{eq:exact_distance_gaussians} is proved in~\cite{MR752258},
    but we will include a sketch of the proof in the simpler case
    where $\mat \Sigma_1, \mat \Sigma_2 \succ 0$ (the proof of the general case requires an additional step)
    for the reader's convenience and because we will employ the intermediate inequality~\eqref{eq:general_lowerbound_wasserstein} below.
    We will see that, by taking an appropriate singular value decomposition,
    the proof presented in the aforementioned paper can be slightly simplified.
    The key idea is to notice that the covariance matrix of the optimal transference plan
    (a probability measure on $\real^d \times \real^d$) must have the form
    \[
        \mat \Sigma =
        \begin{pmatrix}
            \mat \Sigma_1 & \mat X \\
            \mat X^T & \mat \Sigma_2
        \end{pmatrix},
    \]
    and that the Wasserstein distance is given by $\eucnorm{\vect \mu_1 - \vect \mu_2}^2 + \trace(\mat \Sigma_1 + \mat \Sigma_2 - 2 \mat X)$.
    Using Schur's complement,
    and denoting the squared Wasserstein distance on the left-hand side of \cref{eq:exact_distance_gaussians} by $W^2$ for short,
    we deduce
    \[
    W^2 \geq \eucnorm{\vect \mu_2 - \vect \mu_1}^2 + \min_{\mat X} \trace(\mat \Sigma_1 + \mat \Sigma_2 - 2 \mat X) \qquad \text{subject to }\mat \Sigma_2 - \mat X^T \mat \Sigma_1^{-1} \mat X \succeq 0.
    \]
    (The infimum is attained because the admissible set is compact.)
    By polar decomposition of $\mat \Sigma_1^{-1/2} \mat X$,
    it is possible to write $\mat X = \mat \Sigma_1^{1/2} \mat Q \mat S^{1/2}$,
    for an orthogonal matrix $\mat Q$ and a symmetric positive-semidefinite matrix $\mat S^{1/2}$.
    Since $\mat Q$ does not appear in the constraint,
    and since $\trace(\mat X) = \trace(\mat Q \mat S^{1/2} \mat \Sigma_1^{1/2}) = \trace(\mat Q \mat V_1 \mat D \mat V_2^T) = \trace(\mat V_2^T \mat Q \mat V_1 \mat D)$,
    where $\mat V_1^T \mat D \mat V_2$ is the singular value decomposition of $\mat S^{1/2} \mat \Sigma_1^{1/2}$,
    is clearly maximized when $\mat V_2^T \mat Q \mat V_1 = \mat I$ with maximal value $\trace(\mat D)$,
    we deduce
    \begin{align*}
        W^2
        & \geq \eucnorm{\vect \mu_2 - \vect \mu_1}^2 + \trace(\mat \Sigma_1 + \mat \Sigma_2)
        - 2 \max_{\mat S} \trace \left((\mat \Sigma_1^{1/2} \mat S \mat \Sigma_1^{1/2})^{1/2}\right) \mat  \qquad \text{subject to }\mat \Sigma_2 - \mat S \succeq 0,
    \end{align*}
    where the maximum is taken over all symmetric positive-semidefinite matrices.
    Here we employed that $\trace (\mat D) = \trace((\mat V_2^T \mat D^2 \mat V_2)^{1/2}) = \trace((\mat \Sigma_1^{1/2} \mat S \mat \Sigma_1^{1/2})^{1/2})$.
    Since the matrix square root is monotonous over the cone of positive semi-definite matrices,
    and since clearly $\mat \Sigma_1^{1/2} \mat S \mat \Sigma_1^{1/2} \preceq \mat \Sigma_1^{1/2} \mat \Sigma_2 \mat \Sigma_1^{1/2}$  on the set of admissible $\mat S$
    (that is, congruence preserves the order~$\preceq$),
    we conclude that the optimum is reached when $\mat S = \mat \Sigma_2$,
    which leads to
    \begin{equation}
        \label{eq:general_lowerbound_wasserstein}%
        W^2 \geq \eucnorm{\vect \mu_2 - \vect \mu_1}^2 + \trace \left(\mat \Sigma_1 + \mat \Sigma_2 - 2(\mat \Sigma_1^{1/2} \mat \Sigma_2 \mat \Sigma_1^{1/2})^{1/2}\right).
    \end{equation}
    Considering the following transportation map,
    \[
        T: x \mapsto \vect \mu_2 + \mat \Sigma_1^{-1/2} \, (\mat \Sigma_1^{1/2} \mat \Sigma_2 \mat \Sigma_1^{1/2} )^{1/2} \, \mat \Sigma_1^{-1/2} (x - \vect \mu_1),
    \]
    we notice that the lower bound is in fact attained for Gaussian densities.
    Indeed, it is simple to check that $T_{\#} \mathcal N(\vect \mu_1, \mat \Sigma_1) = \mathcal N(\vect \mu_2, \mat \Sigma_2)$
    and, by a change of variable,
    \begin{align*}
        \int \eucnorm{x - Tx}^2 \, g_{\vect \mu_1, \mat \Sigma_1}(x) \, \d x &= \int \eucnorm{\vect \mu_1 - \vect \mu_2 + x - \mat \Sigma_1^{-1/2} \, (\mat \Sigma_1^{1/2} \mat \Sigma_2 \mat \Sigma_1^{1/2} )^{1/2} \, \mat \Sigma_1^{-1/2} x}^2 \, g_{\vect 0, \mat \Sigma_1}(x) \, \d x, \\
                                                                   &= \eucnorm{\vect \mu_1 - \vect \mu_2}^2 + \trace (\mat \Sigma_1) + \trace (\mat \Sigma_2) - 2 \trace \left((\mat \Sigma_1^{1/2} \mat \Sigma_2 \mat \Sigma_1^{1/2})^{1/2}\right),
    \end{align*}
    where we employed the notation $g_{\vect \mu, \mat \Sigma} = g(\dummy, \vect \mu, \mat \Sigma)$ for short.
\end{proof}
\begin{remark}
\label{remark:lower_bound_wasserstein_general}%
We remark that \cref{eq:general_lowerbound_wasserstein} is in fact true for any probability measures with positive-definite covariance matrices,
as Gaussianity had not entered the proof at that point.
It is possible to show that this inequality too holds for degenerate covariant matrices,
see~\cite[Theorem 2.1]{MR1127323}.
\end{remark}

\begin{lemma}
    [Bounds on the Wasserstein distance between Gaussians]\label{lemaux}
    Consider two Gaussians $\mathcal N(\vect \mu_1, \mat \Sigma_1)$ and $\mathcal N(\vect \mu_2, \mat \Sigma_2)$.
    Denoting the Wasserstein distance between them by $W$ for convenience,
    it holds
    \begin{equation}
        \label{eq:bound_distance_gaussians}%
        \frac{1}{2}\eucnorm{\mat \Sigma_1^{1/2} - \mat \Sigma_2^{1/2}}_F^2 \leq W^2 - \abs{\vect \mu_2  - \vect \mu_1}^2 \leq \eucnorm{\mat \Sigma_1^{1/2} - \mat \Sigma_2^{1/2}}_F^2.
    \end{equation}
\end{lemma}
\begin{proof}
The first inequality in \cref{eq:bound_distance_gaussians} can be rewritten as
\[
     \trace (\mat \Sigma_1^{1/2} \mat \Sigma_2 \mat \Sigma_1^{1/2})^{1/2}
     \leq \frac{1}{4} \trace (\mat \Sigma_1 + \mat \Sigma_2 + 2 \mat \Sigma_1^{1/2} \mat \Sigma_2^{1/2})
    = \frac{1}{4} \eucnorm{\mat \Sigma_1^{1/2} + \mat \Sigma_2^{1/2}}[F]^2
\]
or, equivalently,
\[
    \eucnorm{\mat \Sigma_1^{1/2} \mat \Sigma_2^{1/2}}[s_1]
    = \sum_{j} s_j(\mat \Sigma_1^{1/2} \mat \Sigma_2^{1/2})
    \leq \frac{1}{4} s_j \left((\mat \Sigma_1^{1/2} + \mat \Sigma_2^{1/2})^2 \right)
    = \frac{1}{4} \eucnorm{(\mat \Sigma_1^{1/2} + \mat \Sigma_2^{1/2})^2}[s_1],
\]
where $s_j(\dummy)$ is the $j$-th singular value and
$\eucnorm{\dummy}[s_1]$ denotes the Schatten matrix norm with $p = 1$,
defined as the sum of the singular values of its argument.
This inequality follows from the general arithmetic mean/geometric mean inequality,
valid for any unitarily invariant matrix norm and any positive matrices in place of $\mat \Sigma_1^{1/2}$ and $\mat \Sigma_2^{1/2}$,
that is the subject of~\cite{MR1751140}.
To obtain the second inequality in \cref{eq:bound_distance_gaussians},
we employ the standard Araki--Lieb--Thirring inequality with $r = 1/2$ and $q = 1$,
\[
    \trace \left((\mat \Sigma_1^{1/2} \mat \Sigma_2 \mat \Sigma_1^{1/2})^{1/2}\right) \geq \trace \left(\mat \Sigma_1^{1/4} \mat \Sigma_2^{1/2} \mat \Sigma_1^{1/4}\right) = \trace \left(\mat \Sigma_1^{1/2} \mat \Sigma_2^{1/2}\right),
\]
which concludes the proof.
\end{proof}
\begin{remark}
It is in fact possible to recover the second inequality in the bound~\eqref{eq:bound_distance_gaussians} without having recourse to the Araki--Lieb--Thirring inequality,
by simply using the (nonsymmetric) transportation map $x \mapsto \vect \mu_2 + \mat \Sigma_2^{1/2} \, \mat \Sigma_1^{-1/2} (x - \vect \mu_1)$
to obtain an upper-bound for the Wasserstein distance.
\end{remark}

\begin{proposition}
    \label{proposition:stability_wasserstein_gaussian}
    Let $f^1$ and $f^2$ be two Gaussian solutions of \cref{eq:mean-field_equation},
    associated with (Gaussian) initial conditions $f^1_0$ and $f^2_0$, respectively.
    Under the assumptions of \cref{lemma:moments_convergence_noise}, it holds that
    \begin{equation}
        \label{eq:stability_wasserstein_gaussian}%
        W_2(f^1_t, f^2_t) \leq C(1 + mM + m^4 M^{7/2} R) \, \frac{\e^{- \sigma t}}{\sqrt{\alpha(t)}^{1+ \lfloor1\wedge\sigma\rfloor}} \, W_2(f^1_0, f^2_0),
    \end{equation}
    where $C$ is a constant that depends only on the dimension $d$
    and $\alpha(t)$ is given by \cref{eq:definition_alpha}.
\end{proposition}
\begin{proof}
Combining the moment bounds~\eqref{eq:decrease_covariance_square_root} and~\eqref{eq:decrease_first_moment} with \cref{eq:bound_distance_gaussians},
and denoting the Wasserstein distance on the left-hand side of \cref{eq:stability_wasserstein_gaussian} by $W$ for short,
we obtain
\begin{align*}
    W^2
    & \leq \eucnorm*{\mat C_1(t)^{1/2} - \mat C_2(t)^{1/2}}[F]^2 + \eucnorm*{\vect \delta_2(t) - \vect \delta_1(t) }^2 \\
    & \leq C_R^2 M^2 \, m^2 \, \eucnorm*{\mat C_1(0)^{1/2} - \mat C_2(0)^{1/2}}[F]^2 \, \left(\frac{\e^{-2 \sigma t}}{\alpha(t)^2} \right) \\
    & \qquad + \left( 2 \, m M \, \eucnorm*{\vect \delta_1(0) - \vect \delta_2(0)}[2]^2  + m^8 M^6 R^2 \, \eucnorm{\mat C_2(0) - \mat C_1(0)}[F]^2 \right) \, \left(\frac{\e^{-2 \sigma t}}{\alpha(t)} \right) \\
    & \leq \big( 2(C_R m^2 M^2 + m M) \, W_2(f^1_0, f^2_0)^2 + m^8 M^6 R^2 \, \eucnorm{\mat C_2(0) - \mat C_1(0)}[F]^2\big) \, \frac{\e^{- 2 \sigma t}}{\alpha(t) \wedge \alpha(t)^2}.
\end{align*}
Employing \cref{lemma:second_technical_result},
which generalizes the inequality
\[
    \forall a, b \geq 0: \qquad |a - b| = |\sqrt{a} + \sqrt{b}| |\sqrt{a} - \sqrt{b}| \leq 2 \max(\sqrt{a}, \sqrt{b}) |\sqrt{a} - \sqrt{b}|
\]
to symmetric positive semi-definite matrices,
and using \cref{eq:bound_distance_gaussians} again,
we finally obtain
\[
    W_2(f^1_t, f^2_t)^2 \leq  (2 C_R m^2 M^2 + 2 m M + C m^8 M^7 R^2) \, W_2(f^1_0, f^2_0)^2  \, \frac{\e^{- 2 \sigma t}}{\alpha(t) \wedge \alpha(t)^2},
\]
which leads to our claim.
\end{proof}

To prove a more general stability result,
we will combine the ideas of \cref{proposition:convergence_over_manifold_same_covariance} and \cref{proposition:stability_wasserstein_gaussian}.
Additionally, we will need the following lemma.
\begin{lemma}
    [Wasserstein distance between linearly transformed densities]
    \label{lemma:wasserstein_dilated_rotated}%
    Let $\mat A, \mat B \in \real^{d \times d}$ be nonsingular, possibly nonsymmetric matrices,
    and let $f$ be a probability measure with finite second moment, $f \in \mathcal P_2(\real^d)$.
    Then it holds that
    \begin{equation}
        \label{eq:wasserstein_linear_map}%
        W_2(\mat A_\sharp f, \mat B_\sharp f) \leq \eucnorm{\mat A - \mat B}[2]  \, \sqrt{\trace (\mathcal C(f)) + \eucnorm{\mathcal M(f)}^2} ,
    \end{equation}
    where $\mathcal M(f)$ and $\mathcal C(f)$ are the first and second moments of $f$:
    \[
        \mathcal M(f) = \int x \, f(\d x), \qquad
        \mathcal C(f) = \int \left(x - \mathcal M(f) \right) \otimes \left(x - \mathcal M(f) \right) \, f(\d x).
    \]
\end{lemma}
\begin{proof}
    Let us consider the transference plan $\gamma = (\mat A \times \mat B)_\sharp f$,
    which clearly has the required marginals.
    (Here $A \times B$ is the operator $x \mapsto (Ax, Bx)$.)
    We calculate, by a change of variable,
    \begin{align*}
        \int \!\!\! \int_{\real^{d} \times \real^d} \eucnorm{x - y}^2 \, \gamma(\d x \, \d y) & = \int_{\real^{d}} \eucnorm{\mat A x - \mat B x}^2 \, f(\d x) \leq \int_{\real^{d}} \eucnorm{\mat A - \mat B}[2]^2 \, \eucnorm{x}^2 \, f(\d x),
    \end{align*}
    which directly leads to the conclusion.
\end{proof}

\begin{proposition}
    \label{proposition:general_contraction_property}%
    Let $f^1$ and $f^2$ be two solutions of the nonlinear nonlocal mean field equation~\eqref{eq:mean-field_equation} with linear forward model $G$.
    Under the assumptions of \cref{lemma:moments_convergence_noise}, it holds that
    \begin{align}
        \label{eq:general_contraction_property}%
        W_2( f^1_t, f^2_t ) \leq  C(1 + m^4 M^4 + m^4 M^{7/2} R) \, \frac{\e^{- \sigma t}}{\sqrt{\alpha(t)}^{1+ \lfloor1\wedge\sigma\rfloor}} \,  W_2( f^1_0, f^2_0),
    \end{align}
    where $\alpha(t)$ is given by \cref{eq:definition_alpha}.
\end{proposition}
\begin{proof}
    Let us denote the fundamental matrices associated with the two solutions by $\mat U_i(s, t)$, $i = 1, 2$.
    Our starting point will be \cref{eq:intermediate_expression_convolution},
    rewritten in a such a way that the Gaussian densities are centered at zero:
    \begin{align}
        \label{eq:starting_point_general_proof}%
        f^i(\vect u, t) =  \int_{\real^{d}} \frac{f_0^i\left(\mat U_i(0, t)^{-1} (\vect w + \vect u - \vect u_0) + \vect u_0\right)}{\det (\mat U_i(0, t))} \, g(\vect w; 0, \mat \Sigma_i(t)) \, \d \vect w, \qquad i = 1, 2.
    \end{align}
    Introducing new functions $\hat f^i(\vect u, t) := f^i(\vect u + \vect u_0, t)$ and $\hat f^i_0(\vect u) = f^i_0(\vect u + \vect u_0)$ for convenience,
    we obtain the simpler expression
    \begin{align*}
        \hat f^i(\vect u, t) =  \int_{\real^{d}} \frac{\hat f_0^i\left(\mat U_i(0, t)^{-1} (\vect w + \vect u)\right)}{\det (\mat U_i(0, t))} \, g(\vect w; 0, \mat \Sigma_i(t)) \, \d \vect w, \qquad i = 1, 2.
    \end{align*}
    Since the Wasserstein distance is invariant under translation of its arguments,
    it holds that
    \[
        W_2( f^1_t, f^2_t ) = W_2( \hat f^1_t, \hat f^2_t ), \qquad W_2( f^1_0, f^2_0 ) = W_2( \hat f^1_0, \hat f^2_0 ).
    \]
    In other words, we can assume without loss of generality that  $\vect u_0 = 0$.
    From here on, we will drop the hats in $\hat f^i$ and $\hat f^i_0$ for notational convenience.
    Let us now introduce
    \begin{align*}
        f^{1,2}(\vect u, t) =  \int_{\real^{d}} \frac{f_0^1\left(\mat U_1(0, t)^{-1} (\vect w + \vect u) \right)}{\det (\mat U_1(0, t))} \, g(\vect w; 0, \mat \Sigma_2(t)) \, \d \vect w.
    \end{align*}
    Then, using the triangle inequality, we have
    \begin{align}
        \notag%
        W_2(f^1_t, f^2_t)
        & \leq W_2 \left( f^1_t, f^{1,2}_t \right) + W_2 \left( f^{1,2}_t, f^{2}_t \right).
    \end{align}
    Both terms can be simplified by using the convexity property of the Wasserstein metric,
    leading to the inequality
    \begin{align}
        \notag%
        W_2(f^1_t, f^2_t)
        &\leq W_2 \left( g(\dummy; 0, \mat \Sigma_1(t)), g(\dummy; 0, \mat \Sigma_2(t)) \right) \\
        \label{eq:wasserstein_bound_split_in_two}%
        &\quad + W_2 \left( \frac{f_0^1\left(\mat U_1(0, t)^{-1} \dummy \right)}{\det (\mat U_1(0, t))}, \frac{f_0^2\left(\mat U_2(0, t)^{-1} \dummy \right)}{\det (\mat U_2(0, t))} \right).
    \end{align}
    Using \cref{eq:expression_sigma} and employing the triangle inequality again for the second term,
    we obtain
    \begin{align*}
        & W_2(f^1_t, f^2_t)
        \leq \, (1- \e^{-2 \sigma t}) \, W_2 \left( g(\dummy; 0, \mat C_1(t)), g(\dummy; 0,  \mat C_2(t)) \right) \\
        & \quad + W_2 \left( \frac{f^1_0\left(\mat U_1(0, t)^{-1} \dummy \right)}{\det (\mat U_1(0, t))}, \frac{f^1_0\left(\mat U_2(0, t)^{-1} \dummy \right)}{\det (\mat U_2(0, t))} \right)
        + W_2 \left( \frac{f^1_0\left(\mat U_2(0, t)^{-1} \dummy \right)}{\det (\mat U_2(0, t))}, \frac{f^2_0\left(\mat U_2(0, t)^{-1} \dummy \right)}{\det (\mat U_2(0, t))} \right).
    \end{align*}
    Employing \cref{eq:bound_distance_gaussians} for the first term,
    \cref{lemma:wasserstein_dilated_rotated} for the second,
    and \cref{lemma:wasserstein_linear} for the third,
    we deduce
    \begin{align*}
        W_2(f^1_t, f^2_t)
        \leq & \, (1- \e^{-2 \sigma t}) \, \eucnorm*{\mat C_1(t)^{1/2} - \mat C_2(t)^{1/2}}[F] \\
             & + \eucnorm*{\mat U_1(0, t) - \mat U_2(0, t)}[2] \, \sqrt{\trace{\mat C_1(0)} + \eucnorm{\vect \delta_1(0)}^2}
             + \eucnorm{\mat U_2(t) \mat U_2(t)^T}[2] \, W_2 (f^1_0, f^2_0).
    \end{align*}
    Employing \cref{eq:decrease_covariance_square_root} for the first term,
    \cref{eq:contraction_U} and \cref{eq:technical_result_1d} for the second,
    \cref{eq:bound_covariance_u} for the third,
    and \cref{remark:lower_bound_wasserstein_general} to bound $\eucnorm{\mat C_1(t)^{1/2} - \mat C_2(t)^{1/2}}[F]$ from above by the Wasserstein distance,
    we finally obtain
    \begin{align*}
        W_2(f^1_t, f^2_t)
        \leq & \, C (m M + m^4 M^4 + m^4 M^{7/2} R + \sqrt{m M}) \, \frac{\e^{-\sigma t}}{\sqrt{\alpha(t)} \wedge \alpha(t)} \, W_2 (f^1_0, f^2_0),
    \end{align*}
    which concludes the proof.
\end{proof}
We note that, strictly speaking,
\cref{proposition:general_contraction_property} is not a generalization of \cref{proposition:stability_wasserstein_gaussian}
because the constant on the right-hand side of \cref{eq:general_contraction_property} contains the term $m^4 M^4$,
which was not present in \cref{eq:stability_wasserstein_gaussian}.
\begin{remark}
In the case $\sigma = 0$,
assuming without loss of generality that $\vect u_0 = 0$,
we have the following simpler expression instead of \cref{eq:starting_point_general_proof}:
\begin{align*}
f^i(\vect u, t) =  \frac{f_0^i\left(\mat U_i(0, t)^{-1} (\vect u)\right)}{\det (\mat U_i(0, t))}, \qquad i = 1, 2,
\end{align*}
so we directly obtain \cref{eq:wasserstein_bound_split_in_two} without the first term on the right-hand side.
\end{remark}
\begin{remark}
\Cref{proposition:general_contraction_property} can be proved with a probabilistic approach too,
although with slightly different constants on the right-hand side.
Since the proof is very similar in spirit to the one given above,
we will not present it here.
\end{remark}

\begin{remark}
Notice that, in contrast to~\cite[Proposition 2]{2019arXiv190308866G},
the rate of decay shown in \cref{proposition:general_contraction_property} does not depend on $\mat B$,
i.e., on the Hessian of $\Phi_R$. More importantly, the rate of decay for $\sigma>0$ is sharp.
In order to check this,
note first that the mean $\vect\delta (t)$ decays as $\e^{-\sigma t}$ because~\eqref{eq:equation_convenient_u} implies
$$
\eucnorm{\vect\delta (t)}_{\mat C(t)} = \e^{-\sigma t} \eucnorm{\vect\delta (0)}_{\mat C(0)} \qquad \mbox{and thus,} \qquad \eucnorm{\vect\delta (t)} \geq \frac{1}{mM}\frac{\e^{-\sigma t}}{\alpha(t)} \eucnorm{\vect\delta (0)}.
$$
On the other hand,
since the first inequality in \eqref{eq:bound_distance_gaussians} in \cref{lemaux} holds for general probability measures,
then
$$
W_2( f_t, f_\infty ) \geq \eucnorm{\vect\delta (t)}
$$
for any solution $f_t$ of \eqref{eq:mean-field_equation}, with $f_\infty$ being the Gaussian equilibrium.
\end{remark}

\paragraph{Acknowledgments}%
The authors are grateful to Giuseppe Visconti, Grigorios A. Pavliotis, Andrei Velicu, Franca Hoffmann and Andrew Stuart for useful suggestions.
JAC and UV were partially supported by the EPSRC grant number  EP/P031587/1.

\appendix
\section{Auxiliary technical results}%
\label{sec:auxiliary_technical_results}

\begin{lemma}
    [A concavity inequality]
    \label{lemma:technical_result_concavity_matrix_square_root}%
    Let $\mat M_1$, $\mat M_2$ and $\mat M$ be symmetric, positive-semidefinite matrices in $\real^{d\times d}$.
    Then it holds that
    \begin{equation}
        \label{eq:technical_result_concavity_matrix_square_root}%
        \eucnorm{(\mat M + \mat M_1)^{1/2} - (\mat M + \mat M_2)^{1/2}}[F] \leq C_R(d) \eucnorm{\mat M_1^{1/2} - \mat M_2^{1/2}}[F],
    \end{equation}
    for a constant $C_R$ that depends only on $d$.
\end{lemma}
\begin{proof}
    The statement is obvious in one dimension.
    For the general case,
    we start by showing the statement for the metric
    \begin{equation}
        \label{eq:technical_result_1d}
        d(\mat M_1, \mat M_2)
        = \sup_{\eucnorm{x} = 1} \abs{ \vphantom{\Big(}\! \eucnorm{\mat M_1 x} - \eucnorm{\mat M_2 x} }
        = \sup_{\eucnorm{x} = 1} \abs{ \sqrt{x^T \mat M_1^2 x} - \sqrt{x^T \mat M_2^2 x} },
    \end{equation}
    and then we show that this metric is equivalent to the that induced by the Frobenius norm (or any other matrix norm) on the space of symmetric positive-semidefinite matrices.
    To complete the first part, we expand \cref{eq:technical_result_1d} and use the one-dimensional version of this lemma:
    \begin{align*}
        d((\mat M + \mat M_1)^{1/2}, (\mat M + \mat M_2)^{1/2}) &= \sup_{ \eucnorm{x} = 1} \abs{ \sqrt{x^T (\mat M + \mat M_1) x} - \sqrt{x^T (\mat M + \mat M_2) x} } \\
                                                &= \sup_{ \eucnorm{x} = 1} \abs{ \sqrt{x^T\mat M x + x^T\mat M_1 x} - \sqrt{x^T \mat M x + x^T \mat M_2 x} } \\
                                                &\leq \sup_{ \eucnorm{x} = 1} \abs{\sqrt{x^T \mat M_1 x} - \sqrt{x^T \mat M_2 x}} = d(\mat M_1^{1/2}, \mat M_2^{1/2}).
    \end{align*}
    To complete the second part, we must show that there exist constants $C_1$ and $C_2$ such that
    \[
        \forall \mat M_1, \mat M_2 \succcurlyeq 0, \qquad  C_1 \eucnorm{\mat M_1 - \mat M_2}[F] \leq d(\mat M_1, \mat M_2) \leq C_2 \eucnorm{\mat M_1 - \mat M_2}[F].
    \]
    The first inequality is proved in~\cite[Lemma C.1]{abdulle2017spectral}.
    The second inequality follows after taking the supremum (over the sphere $\eucnorm{x} = 1$) in the following equation,
    where we employ the triangle inequality:
    \[
        \abs{\vphantom{\Big(} \! \eucnorm{\mat M_1 x} - \eucnorm{\mat M_2 x} } \leq \eucnorm{\mat M_1 x - \mat M_2 x} \leq \eucnorm{\mat M_1 - \mat M_2}[2] \eucnorm{x}.
    \]
    This completes the proof.
\end{proof}

Using the same trick, of passing to the equivalent distance $d(\dummy, \dummy)$,
we can show the following.
\begin{lemma}
    \label{lemma:second_technical_result}%
    Let $\mat M_1$, $\mat M_2$ be symmetric, positive-semidefinite matrices in $\real^{d\times d}$.
    It holds that
    \begin{equation}
        \label{eq:second_technical_result}%
        \eucnorm*{\mat M_1 - \mat M_2}[F] \leq C(d) \, \max(\eucnorm*{\mat M_1^{1/2}}[F], \eucnorm*{\mat M_2^{1/2}}[F]) \, \eucnorm*{\mat M_1^{1/2} - \mat M_2^{1/2}}[F],
    \end{equation}
    for a constant $C$ that depends only on $d$.
\end{lemma}
\begin{proof}
    In one dimension, the statement follows from the equation
    \[
        \forall m_1, m_2 \geq 0: \qquad \abs{m_1 - m_2} = \abs{\sqrt{m_1} - \sqrt{m_2}} \, \abs{\sqrt{m_1} + \sqrt{m_2}}.
    \]
    We can then show pass to $d(\dummy, \dummy)$ as follows:
    \begin{align*}
        \eucnorm*{\mat M_1^{1/2} - \mat M_2^{1/2}}[F]
        & \geq C \, \sup_{\eucnorm{x} = 1} \abs*{ \sqrt{x^T \mat M_1 x} - \sqrt{x^T \mat M_2 x}} \\
        & = C \, \sup_{ x \in S } \frac{\abs*{x^T \mat M_1 x - x^T \mat M_2 x}}{\sqrt{x^T \mat M_1 x} + \sqrt{x^T \mat M_2 x}} \geq C \, (\eucnorm*{\mat M_1^{1/2}}[2] + \eucnorm*{\mat M_2^{1/2}}[2])^{-1} \, \eucnorm{\mat M_1 - \mat M_2}[2],
    \end{align*}
    where $S := \{x: \eucnorm{x} = 1, x^T (\mat M_1 + \mat M_2) x > 0 \}$.
    This leads to the statement after rearranging.
\end{proof}

\bibliography{references}
\bibliographystyle{abbrv}

\end{document}